%% file: main.tex
\documentclass[11pt]{amsart}
\usepackage{graphicx}
\usepackage{amssymb}
\usepackage{mathtools}
\usepackage[alphabetic]{amsrefs}
\usepackage{tikz}
\usepackage{tikz-3dplot}
\usepackage{float}
\usepackage{fullpage}
\title{A Cayley--Menger formula for the earth mover's simplex}

\author{William Q. Erickson}
\address{
William Q.~Erickson\\
Department of Mathematics\\
Baylor University \\ 
One Bear Place \#97328\\
Waco, TX 76798} 
\email{Will\_Erickson@baylor.edu}

\theoremstyle{plain}
\newtheorem{theorem}{Theorem}[section]
\newtheorem{prop}[theorem]{Proposition}

\newtheorem{cor}[theorem]{Corollary}

\theoremstyle{definition}
\newtheorem{definition}[theorem]{Definition}
\newtheorem{rem}[theorem]{Remark}
\newtheorem{example}[theorem]{Example}
\newtheorem{problem}[theorem]{Problem}

\newcommand{\F}{\mathcal{F}}
\renewcommand{\H}{\mathcal{H}}
\newcommand{\EMD}{{\rm EMD}}
\newcommand{\X}{\mathcal{X}}
\newcommand{\smtriangle}{\,\scalebox{.75}{$\triangle$} \,}
\newcommand{\Min}{{\rm Min}}
\newcommand{\Med}{{\rm Med}}
\newcommand{\Maj}{{\rm Maj}}
\newcommand{\Vol}{{\rm Vol}}



\begin{document}

\begin{abstract}
   The earth mover's distance (EMD) is a well-known metric on spaces of histograms; roughly speaking, the EMD measures the minimum amount of work required to equalize two histograms.
   The EMD has a natural generalization that compares an arbitrary number of histograms; in this case, the EMD can be viewed as hypervolume in $d$ dimensions, where the histograms are vertices of a $d$-simplex.
   For $d=2$, it is known that the EMD between three histograms equals half the sum of the pairwise EMDs --- a sort of Heron's formula for histograms, but where the area equals the semiperimeter.
   In this paper, by introducing an object we call the earth mover's simplex, we prove two generalizations of this Heron-like formula in arbitrary dimension: 
   the first (a sort of Cayley--Menger formula) expresses the EMD in terms of the edge lengths (the pairwise EMDs), the second in terms of the facets (EMDs excluding one histogram).
    
\end{abstract}

\maketitle

\section{Introduction}

\subsection{Cayley--Menger formula for geometric simplices}

Among the more famous results of antiquity is \emph{Heron's formula} for the area of a triangle $\triangle ABC$ in terms of its side lengths $a,b,c$:
\begin{equation}
    \label{Heron}
    \text{area of }\triangle ABC = \sqrt{s(s-a)(s-b)(s-c)},
\end{equation}
where $s = (a+b+c)/2$ is the semiperimeter.
Although named for Heron of Alexandria, who recorded it in his \textit{Metrica} in the first century, this result was likely known much earlier, perhaps even to Archimedes.

A natural question is whether Heron's formula generalizes to a formula for (hyper)volume in higher dimensions;
the answer depends on what exactly we wish to generalize, since in the original formula~\eqref{Heron} in two dimensions, the side lengths $a,b,c$ have both dimension 1 \emph{and} codimension 1.
On one hand, if we view them as having dimension 1, then the problem is to express the volume of a $d$-dimensional simplex $\Delta$ in terms of its edge lengths (i.e., the volumes of its 1-dimensional faces).
The answer to this problem is the~\emph{Cayley--Menger determinant} below:
\begin{equation}
    \label{CM}
    \text{volume of } \Delta = \frac{(-1)^{d+1}}{(d!)^2 2^d} \cdot \det \! \big[\ell^2_{ij}\big]_{i,j=0}^{d+1},
\end{equation}
where $\ell_{ij}$ is the length of the edge between the $i$th and $j$th vertices (and $\ell_{i,d+1} = \ell_{d+1,i} \coloneqq 1$ for all $i \leq d$, and $\ell_{d+1,d+1} \coloneqq 0$).  
On the other hand, if we view the sides of a triangle as having \emph{co}dimension 1, then the question is whether the volume of $\Delta$ can be expressed only in terms of its ``surface area'' (i.e., the volumes of its facets).
In this case, the answer is no, since even for the tetrahedron it is clear that the four facet areas do not determine the volume.
Hence any generalization of Heron's formula with respect to the facets must require additional information.

\subsection{Statement of the problem}

The problem in this paper is to find an analogue of the Cayley--Menger formula~\eqref{CM}, in the very specific context of the \emph{earth mover's distance} (EMD).
The EMD is a metric on spaces of histograms, or probability distributions, and is essentially synonymous with the 1-Wasserstein distance and the Mallows distance~\cite{Bickel}: 
roughly speaking, the EMD measures the minimum amount of work required to equalize two histograms, where ``work'' is defined by the geometry of the feature space of the histograms.
Although outside the scope of the present paper, the EMD can be viewed as the solution to the \emph{transportation problem} (the founding problem of optimal transport theory), and is central to a burgeoning range of important applications in mathematics, physics, and the social sciences; see Villani's comprehensive reference~\cite{Villani} for further details.

In this paper, we consider histograms in which the bins are the integers $1, \ldots, n$ on the number line.
It is natural to generalize the EMD in order to compare an arbitrary number of histograms $h_0, \ldots, h_d$, rather than only two at a time:
in this case, the EMD measures the minimum amount of work required to equalize all $d+1$ histograms.
In the special case $d=2$, it was observed~\cite{Erickson20}*{Prop.~5} that there is a simple linear relationship between the EMD of three histograms, on one hand, and the three pairwise EMDs, on the other hand:
\begin{equation}
    \label{observation 3 and 2}
    \EMD(h_0, h_1, h_2) = \frac{1}{2}\big[ \EMD(h_1, h_2) + \EMD(h_0, h_2) + \EMD(h_0, h_1)\big].
\end{equation}
This identity instantly reminds one of Heron's formula, where the histograms $h_i$ are vertices of a triangle, their EMD is the area, and the pairwise EMDs are the side lengths.
In this EMD setting, however, the Heron analogue~\eqref{observation 3 and 2} is much simpler than the Euclidean version~\eqref{Heron}, since in~\eqref{observation 3 and 2} the area simply equals the semiperimeter.
It is also shown in~\cite{Erickson20} that a simple formula of the kind in~\eqref{observation 3 and 2}, expressing overall EMD exclusively in terms of the pairwise EMDs, cannot exist for $d>2$.
The purpose of this paper, then, is to find the correct generalization of~\eqref{observation 3 and 2} in arbitrary dimensions.

\subsection{Overview of results}

As our primary tool, we introduce a combinatorial object $\Delta$ we call the \emph{earth mover's (EM) simplex} (see Definition~\ref{def:EM simplex}).
We endow $\Delta$ with a ``volume'' in such a way that, if the vertices of $\Delta$ are taken to be the cumulative histograms of $h_0, \ldots, h_d$, then the volume of $\Delta$ equals $\EMD(h_0, \ldots, h_1)$.
In this way, we translate the problem into that of expressing the volume of $\Delta$ in terms of its edge lengths.
In fact, this ``volume'' is merely the first in a sequence of generalized volumes we define on $\Delta$; it turns out that the second generalized volume measures the failure of the obvious analogue of~\eqref{observation 3 and 2}.
Hence our main result (Theorem~\ref{thm:CM}), a sort of Cayley--Menger formula for EM simplices, takes the especially nice form
\[
    \Vol(\Delta) = \frac{1}{d} \Bigg( \Vol_2(\Delta) + \sum_{\mathclap{\substack{\textup{edges $\mathcal{E}$} \\ \textup{of $\Delta$}}}} \Vol(\mathcal{E}) \Bigg).
\]
We also prove a second volume formula for EM simplices (Theorem~\ref{thm:Vol from SA}),  this time in terms of the surface area ${\rm SA}(\Delta)$:
\[
        \Vol(\Delta) = 
        \frac{1}{d} \left( {\rm{SA}}(\Delta) + \frac{d+1}{2} \cdot |\Med(\Delta)| \right),
    \]
    where $|\Med(\Delta)|$ is the number of elements contained in exactly half of the vertices (which, in an EM simplex, are themselves sets).
    We point out that in even dimensions, we automatically have $|\Med(\Delta)| = 0$; we find it interesting that the dimensional parity affects the EMD in this way, since a similar effect has been observed with regard to its expected value~\cite{Erickson20}*{\S5.4}.

    In Section~\ref{sec:Examples} we present a fully illustrated example of both main theorems, which may be helpful in understanding the definition of an EM simplex.
    We are also hopeful that the results in this paper might suggest a direction for future research by experts in topological data analysis, using the tools of persistent homology; see our closing remarks in Section~\ref{sec:Conclusion}.

\section{The earth mover's distance}

\subsection{EMD between two histograms}

Classically, the \emph{earth mover's distance} (EMD) measures the similarity between two histograms: the EMD is defined to be the minimum amount of work required to transform one histogram into the another.
In this paper, the histograms we consider have bins $1, \ldots, n$, and some fixed number $m$ of data points.
We define one unit of work to be the work required to move one data point by one bin.
We denote a histogram by a lower-case $h = (h(1), \ldots, h(n))$, where $h(i)$ is the number of data points in bin $i$.
We denote the cumulative histogram of $h$ by an upper-case $H = (H(1), \ldots, H(n))$, where  $H(j) = \sum_{i=1}^j h(i)$.
It will be convenient for us to identify $H$ with the set of dots in its dot diagram, as in the following example:
\begin{equation}
    \label{h example}
h = (3,0,1,4,2), \qquad H = (3,3,4,8,10) = 
\begin{tikzpicture}[scale=.2, baseline=(current bounding box.center)]
  \foreach \x/\y in {1/1, 1/2, 1/3, 2/1, 2/2, 2/3, 3/1, 3/2, 3/3, 3/4, 4/1, 4/2, 4/3, 4/4, 4/5, 4/6, 4/7, 4/8, 5/1, 5/2, 5/3, 5/4, 5/5, 5/6, 5/7, 5/8, 5/9, 5/10} {
    \fill (\x, \y) circle (10pt);
  }
  
  \draw[->] (0,0) -- (6,0) node[right] {};
  \draw[->] (0,0) -- (0,10) node[above] {};
  
\end{tikzpicture}
\end{equation}
Note that in the language of combinatorics, each histogram can be identified with a unique \emph{(weak) composition} of $m$ into $n$ parts, while a cumulative histogram is a \emph{partition} of $\sum_i H(i)$ into $n$ parts.
It is customary to identify a partition with its \emph{Ferrers diagram}, which is essentially the dot diagram here (up to rotation by 180 degrees).
By ignoring the last column of $H$, which always contains $m$ dots, it is easy to see \cite{Erickson23}*{\S3} that the set of cumulative histograms is in bijection with the set of Ferrers diagrams fitting inside an $m \times (n-1)$ rectangle.

Given two histograms $h_0, h_1$, it is well known that $\EMD(h_0, h_1)$ equals the $\ell_1$-norm of $H_0 - H_1$, as a difference of functions on $\{1, \ldots, n\}$.  
(See ~\cite{Rabin}*{eqn.~(8) and Fig.~1}; it follows that the EMD is a true metric on the space of histograms.)
From the perspective of~\cite{Erickson23}, this means that 
\begin{equation}
\label{EMD is symm diff d=2}
    \EMD(h_0, h_1) = | H_0 \smtriangle H_1 |,
\end{equation}
where $H_0 \smtriangle H_1 \coloneqq (H_0 \cup H_1) \setminus (H_0 \cap H_1) = (H_0 \setminus H_1) \cup (H_1 \setminus H_0)$ denotes the \emph{symmetric difference}.
In words, the EMD is the number of dots occurring in exactly one of the two cumulative histograms;
see the example in Figure~\ref{fig:example EMD d2}.

\begin{figure}[ht]
    \centering
    \input{EMD_ex}
    \caption{The EMD between histograms $h_0$ and $h_1$ is the size of the symmetric difference $H_0 \smtriangle H_1$. 
    In the right-hand picture, the unfilled dots are not elements of the symmetric difference, since they are in the intersection $H_0 \cap H_1$.
    The minimum of 7 units of work can be attained by transporting the following data points in $h_0$: move two points from bin 1 to bin 2, move one point from bin 3 to bin 2, one from 4 to 2 (which expends two units of work), one from 4 to 3, and one from 4 to 5.}
    \label{fig:example EMD d2}
\end{figure}



\subsection{Generalized EMD}

The EMD has a natural generalization for any number of histograms, rather than only two at a time.
(See, for example,~\cite{Bein} and~\cite{Kline} for computational treatments of the higher-dimensional earth mover's problem in greater generality.)
Given histograms $h_0, \ldots, h_d$, one simply defines $\EMD(h_0, \ldots, h_d)$ to be the minimum amount of work required to transform all of the $h_i$ into a common histogram (allowing data points to be transported within each histogram).

In order to extend the conceptual method of Figure~\ref{fig:example EMD d2} beyond the classical case $d=1$, we require the correct generalization of the symmetric difference.
It turns out that the key property of the symmetric difference (with regard to the EMD) was this: 
for each dot $x \in H_1 \cup H_2$, the symmetric difference measures how far away $x$ is from either belonging to \emph{none} of the $H_i$ or belonging to \emph{all} of the $H_i$.
When $d=1$ this is a binary measurement: if $x$ belongs to neither/both of the $H_i$, then it appears 0 times in the symmetric difference, and otherwise it appears 1 time in the symmetric difference.
For arbitrary $d$, however, generalizing this measurement requires the symmetric difference to be a \emph{multi}set:
the multiplicity of each dot $x$ is either the number of $H_i$ containing $x$, or the number of $H_i$ not containing $x$, whichever is smaller.

Throughout the paper, we write $\uplus$ to denote the union of multisets, and we write $x^m$ to denote a multiset element with multiplicity $m$.  
If $\X$ is a family of sets, then for each $x \in \bigcup_{X \in \X} X$, its \emph{degree}, written as $\deg_\X(x)$, is the number of sets $X \in \X$ such that $x \in X$.

\begin{definition}[\cite{Erickson23}*{Def.~7.5}]
\label{def:GSD}
    Let $\X = \{X_0, \ldots, X_d\}$ be a family of sets.
    The \emph{generalized symmetric difference} is the multiset
    \[
    \blacktriangle(\X) \coloneqq \biguplus_{k=0}^{d+1} \{ x^{\min\{k, \: d-k+1\}} \mid \deg_{\X}(x) = k \}.
    \]
\end{definition}

\begin{prop}[\cite{Erickson23}*{Prop.~7.7}]
    \label{prop:EMC def original}
    Let $h_0, \ldots, h_d$ be histograms, and let $\H = \{H_0, \ldots, H_d\}$ be the family of cumulative histograms.
    Then
    \[
        \EMD(h_0, \ldots, h_d) = |\blacktriangle(\H)|.
    \]
\end{prop}

See Figure~\ref{fig:example EMC d3} for a toy example, which we will revisit in Section~\ref{sec:Examples}.

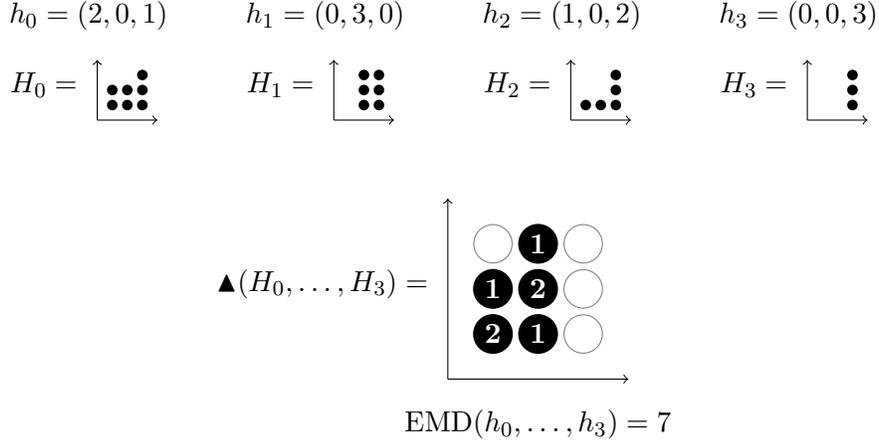
\begin{figure}[t]
    \centering
    \input{EMC_small_ex}
    \input{EMC_USD_small}
    \caption{Example of the generalized EMD, where $d=n=m=3$.
    Using Definition~\ref{def:GSD}, we depict the multiset $\blacktriangle(H_0, \ldots, H_3)$ by labeling each dot with its multiplicity (where the 0's are unfilled dots). 
    In particular, the multiplicity of a dot is $0$, $1$, $2$, $1$, or $0$, depending on whether it is contained in $0$, $1$, $2$, $3$, or $4$ of the $H_i$, respectively.
    (Note that the middle dot in the bottom row has degree 3, and therefore its multiplicity is $\min\{3, \:1\} = 1$.
    Likewise, the dots in the third column have degree 4, and therefore multiplicity $\min\{4,0\} = 0$.)
    The EMD is then obtained via Proposition~\ref{prop:EMC def original}.
    The reader can check that one way to equalize the histograms with only 7 units of work is to transform each of them into $(0,2,1)$.
    }
    \label{fig:example EMC d3}
\end{figure}

\section{The earth mover's simplex}
\label{sec:EM simplex}

In this section, we introduce an abstract structure we call the \emph{earth mover's (EM) simplex}: 
this is a combinatorial simplex whose vertices are finite sets $X_i$, equipped with face labelings $\lambda, \lambda', \lambda'', \ldots$, to be defined below.
The labeling $\lambda'$ in particular induces a notion of ``volume'' that is compatible with the EMD, whenever the vertices are chosen to be cumulative histograms $H_i$.
Hence at the end of this section we will prove the motivating fact that $\EMD(h_0, \ldots, h_d)$ is precisely the volume of the EM simplex on the $H_i$.

\subsection{Standard terminology}
\label{sub:terminology}

Let $\Delta = \Delta(V)$ be the (abstract) $d$-simplex on the vertex set $V = \{v_0, \ldots, v_d\}$.
In other words, $\Delta$ is the power set of $V$, so that the elements (\emph{faces}) of $\Delta$ are simply the subsets $F \subseteq V$.
The \emph{dimension} of a face is one less than its cardinality: hence we write $\dim F \coloneqq \#F - 1$, and in particular we have $\dim \varnothing = -1$.
By definition, $\dim \Delta \coloneqq \dim V = d$.
The \emph{codimension} of a face is $\operatorname{codim} F \coloneqq d - \dim F$.
The 0-dimensional faces $\{v_i\}$ are called \emph{vertices}, and the 1-dimensional faces $\{v_i, v_j\}$ are called \emph{edges}.
The $(d-1)$-dimensional faces $\{v_0, \ldots, \widehat{v}_i, \ldots, v_d\}$ are called \emph{facets} of $\Delta$, and the missing $v_i$ is called the vertex \emph{opposite} the facet.

It is customary, when convenient, to view a face as the simplex it defines.
For example, we call $F$ a \emph{facet} of $G \in \Delta$ whenever $F \subset G$ with $\dim F = \dim G - 1$, although technically it would be more proper to call $F$ a facet of $\Delta(G)$.
In this case, we also say that $G$ is a \emph{cofacet} of $F$.
Clearly the number of facets of $G$ equals its cardinality $\#G$.

We denote the \emph{$k$-skeleton} of $\Delta$ by ${\rm skel}_k(\Delta) \coloneqq \{ F \in \Delta \mid \dim F \leq k\}$.
Finally, the \emph{link} of a face $F$ is the subsimplex of $\Delta$ containing all faces which are disjoint from $F$.
(This is a special case of the standard definition of link in simplicial complexes.)

\subsection{Face labelings}

In our upcoming definition of an EM simplex, the role of vertices $v_i$ will be played by finite sets $X_i$.
Hence let $\X = \{X_0, \ldots, X_d\}$ be a family of finite sets, and from now on, let $\Delta = \Delta(\X)$.
We use the shorthand $\cup\X \coloneqq \bigcup_{X \in \X} X$, and we partition $\cup\X$ into three subsets according to degree:
\begin{align*}
    \Min(\X) &\coloneqq \left\{x \in \cup\X \; \middle| \; \deg_{\X}(x) < \frac{d+1}{2} \right\}, \\
    \Med(\X) &\coloneqq \left\{x \in \cup\X \; \middle| \; \deg_{\X}(x) = \frac{d+1}{2} \right\}, \\
    \Maj(\X) &\coloneqq \left\{x \in \cup\X \; \middle| \; \deg_{\X}(x) > \frac{d+1}{2} \right\}.
\end{align*}
(These are the elements contained in the \emph{minority}, in the \emph{median} number, or in the \emph{majority} of the $X_i$, respectively.)
Note that if $d$ is even, then $\Med(\X) = \varnothing$.
Define a map $\epsilon$ which assigns to each $x \in \cup\X$ a face $\epsilon(x)$ in the ``half-skeleton'' of $\Delta$, according to the memberships of $x$:
\begin{align}
\label{epsilon}
\begin{split}
    \epsilon: \cup\X &\longrightarrow {\rm skel}_{\lfloor (d-1)/2 \rfloor}(\Delta),\\
    x &\longmapsto 
    \begin{cases}
        \{X \in \X \mid x \in X\}, & x \in \Min(\X) \cup \Med(\X),\\
        \{X \in \X \mid x \not\in X\}, & x \in \Maj(\X).
    \end{cases}
    \end{split}
\end{align}

We now define a sequence $\lambda^{(0)}, \lambda^{(1)}, \lambda^{(2)}, \ldots$ of face labelings, as follows.
Begin by defining $\lambda^{(0)} \coloneqq \epsilon^{-1}$, thereby labeling each face $\F \in {\rm skel}_{\lfloor (d-1)/2 \rfloor}(\Delta)$ by its fiber:
\begin{equation}
    \label{lambda}
    \lambda^{(0)}(\F) \coloneqq \{ x \in \cup\X \mid \epsilon(x) = \F \}.
\end{equation}
Then recursively define
\begin{equation}
    \label{lambda i}
    \lambda^{(i+1)}(\F) \coloneqq \biguplus_{\mathclap{\substack{\text{cofacets} \\ \text{$\mathcal{G}$ of $\F$}}}} \lambda^{(i)}(\mathcal{G}).
\end{equation}
Note that $\lambda^{(0)}(\F)$ and $\lambda^{(1)}(\F)$ are multiplicity-free (i.e., they are true sets), whereas for $i \geq 2$, the labels $\lambda^{(i)}(\F)$ can be multisets.
The reader may prefer to imagine constructing the labels $\lambda^{(i+1)}(\F)$ as follows: starting with a clean slate, choose a face $\mathcal{G}$, and add one copy of its previous label $\lambda^{(i)}(\mathcal{G})$ to each facet $\F$ of $\mathcal{G}$.
Doing this for all $\mathcal{G}$ where $\lambda^{(i)}(\mathcal{G})$ exists, the resulting multiset unions are precisely the new labels $\lambda^{(i+1)}(\F)$.
Hence, intuitively, one can regard each successive labeling as spreading $\#\mathcal{G}$ copies of the previous label of $\mathcal{G}$ across the boundary $\partial(\mathcal{G})$, for each face $\mathcal{G}$ in the appropriate skeleton.

Note that the domain of $\lambda^{(i)}$ is ${\rm skel}_{\lfloor (d-1)/2 \rfloor - i}(\Delta)$, and so the nonempty labelings are given by the finite sequence $\lambda^{(0)}, \ldots, \lambda^{(\lceil d/2 \rceil)}$.
From now on, we will use the familiar shorthand $\lambda \coloneqq \lambda^{(0)}$, $\lambda' \coloneqq  \lambda^{(1)}$, and $\lambda'' \coloneqq \lambda^{(2)}$.
These are the only labelings that play a role in this paper, and it will turn out that the notation is intentionally suggestive of derivatives.

\subsection{Volume of the EM simplex}

We now make the two central definitions of this paper:

\begin{definition}[EM simplex]
\label{def:EM simplex}
    Let $\X = (X_0, \ldots, X_d)$ be a family of finite sets.
    Then the \emph{earth mover's (EM) simplex} on $\X$ is the abstract $d$-simplex $\Delta = \Delta(\X)$, equipped with the face labelings $\lambda^{(i)}$ defined in~\eqref{lambda}--\eqref{lambda i}.
\end{definition}

\begin{definition}[Volume of an EM simplex]
\label{def:EM volume}
    Let $\Delta$ be an EM simplex.
    The $i$th \emph{generalized volume} of $\Delta$ is
    \[
        \label{volume}
        \Vol_i(\Delta) \coloneqq \sum_{\F} |\lambda^{(i)}(\F)|,
    \]
    where the nonzero summands range over the faces $\F \in {\rm skel}_{\lfloor(d-1)/2\rfloor -i}(\Delta)$.
    When $i=1$, we will suppress the subscript, and simply call $\Vol(\Delta) \coloneqq \Vol_1(\Delta)$ the \emph{volume} of $\Delta$.
\end{definition}

As mentioned in Section~\ref{sub:terminology}, each face $\F \in \Delta$ determines the EM simplex $\Delta(\F)$, and so for notational clarity we will write $\Vol(\F) \coloneqq \Vol(\Delta(\F))$.

\begin{example}[Volumes of EM simplices]\
\label{ex:volumes}
    
    \begin{enumerate}
    
    \item If $d=0$, then $\lambda^{(i)}$ is just the empty labeling for all $i > 0$.
    Hence $\Vol(\Delta) = 0$, just as we would expect for a 0-dimensional simplex.

    \item \label{ex: edge length fact} If $d=1$, then we have $\lambda(\{X_0\}) = X_0 \setminus X_1$ and $\lambda(\{X_1\}) = X_1 \setminus X_0$ and $\lambda(\varnothing) = X_0 \cap X_1$.
    Then the domain of $\lambda'$ is just ${\rm skel}_{-1}(\Delta) = \{\varnothing\}$, and we have 
    \begin{align*}
        \lambda'(\varnothing) &= (X_0 \setminus X_1) \cup (X_1 \setminus X_0) = X_0 \smtriangle X_1 \\
        \Vol(\Delta) &= | X_0 \smtriangle X_1 |.
    \end{align*}
    This fact about ``edge length'' is crucial to our main result (Theorem~\ref{thm:CM}).

    \item When $d=2$, the process above shows that $\Vol(\Delta) = \#\{x \in \cup\X \mid \deg_\X(x) = 1 \text{ or } 2 \} = |\blacktriangle(\X)|$.
    Note that $\Vol(\Delta)$ does not behave like Euclidean volume: for instance, in the situation $X_0 = X_1 \neq X_2$ where two vertices are equal, we always have nonzero $\Vol(\Delta) = |X_0 \smtriangle X_2|$.

    \item Notice that for $d \leq 2$, we had $\Vol_i(\Delta) = 0$ for all $i > 1$.
    See Section~\ref{sec:Examples} for a three-dimensional example where $\Vol_2(\Delta) \neq 0$.

    \item In any dimension $d$, if all the $X_i$ are identical, then the only nonempty $\lambda(\F)$ is $\lambda(\varnothing) = X_0 = \cdots = X_d$.
    Since $\varnothing$ has no facets, we have $\Vol(\Delta) = 0$.

    \item In any dimension $d$, if the $X_i$ are pairwise disjoint, then 
    $\lambda(\F)$ is nonempty if and only if $\F$ is a vertex, in which case $\lambda(\{X_i\}) = X_i$.
    Hence the only nonempty $\lambda'(\F)$ is $\lambda'(\varnothing) = \cup\X$, meaning that $\Vol(\Delta) = |{\cup\X}|$.
    Likewise, for any face $\F \in \Delta$ we have $\Vol(\F) = |{\cup\F}|$.

    \end{enumerate}
    
\end{example}

The following proposition reveals the reason for choosing notation $\lambda, \lambda', \lambda'', \ldots$ suggestive of derivatives.
Specifically, the generating function for $\#\epsilon(x)$ encodes all of the generalized volumes $\Vol_i(\Delta)$, as its $i$th derivatives evaluated at unity:

\begin{prop}
\label{prop:gen function}
    Let $\Delta = \Delta(\X)$ be an EM simplex, and let $\displaystyle v(t) \coloneqq \sum_{x \in \cup\X} t^{\#\epsilon(x)}$.
    Then for all $i \geq 0$, we have
    \[
        \Vol_i(\Delta) = v^{(i)}(1).
    \]
\end{prop}

\begin{proof}
    We will first show by induction that 
    \begin{equation}
        \label{vi triple sum}
        v^{(i)}(t) = \sum_{x \in \cup\X} \hspace{1ex} \sum_{k \geq 0} \hspace{1ex}\sum_{\substack{\F\in \Delta: \\ \#\F = k}} \Big(\text{multiplicity of $x$ in $\lambda^{(i)}(\F)$}\Big) t^{k},
    \end{equation}
    where the multiplicity is 0 if $\lambda^{(i)}(\F)$ is not defined.
    In the base case $i=0$, we have 
    \[
        v(t) \coloneqq \sum_x t^{\#\epsilon(x)} = 
        \sum_x \sum_k \sum_{\#\F=k} \mathbf{1}_{\lambda(\F)}(x) \cdot t^k,
    \]
    which agrees with~\eqref{vi triple sum} since $\lambda(\F)$ is multiplicity-free.
    Now taking~\eqref{vi triple sum} as our induction hypothesis, we have
    \begin{align*}
        v^{(i+1)}(t) = \frac{d}{dt} \: v^{(i)}(t) &= \sum_x \sum_k \sum_{\#\mathcal{G} = k} k\Big(\text{mult.\ of $x$ in $\lambda^{(i)}(\mathcal{G})$}\Big) t^{k-1} \\
        &= \sum_x \sum_k \sum_{\#\mathcal{G} = k} \sum_{\substack{\text{facets} \\ \text{$\F$ of $\mathcal{G}$}}} \Big(\text{mult.\ of $x$ in $\lambda^{(i+1)}(\F)$}\Big) t^{k-1} \\
        &= \sum_x \sum_k \sum_{\#\F = k-1} \Big(\text{mult.\ of $x$ in $\lambda^{(i+1)}(\F)$}\Big) t^{k-1},
    \end{align*}
    which, upon re-indexing, proves the claim~\eqref{vi triple sum}.
    Evaluating~\eqref{vi triple sum} at $t=1$, we obtain
    \[
        v^{(i)}(1) = \sum_x \sum_k \sum_{\#\F=k} \Big(\text{mult.\ of $x$ in $\lambda^{(i)}(\F)$}\Big) = \sum_{\F \in \Delta} |\lambda^{(i)}(\F)| \eqqcolon \Vol_i(\Delta),
    \]
    as in Definition~\ref{def:EM volume}.
\end{proof}

\begin{cor}
\label{cor:Vol IDs}
    Let $\Delta = \Delta(\X)$ be an EM simplex.
    Then we have 
    \[
        \Vol_i(\Delta) = \sum_{x \in \cup\X} (\#\epsilon(x))_i
    \]
    where $(a)_i \coloneqq a(a-1) \cdots (a-i+1)$ is the Pochhammer symbol for the falling factorial.  
    In particular, we have
    \begin{alignat}{2}
        &\Vol_0(\Delta) &&= |{\cup\X}|, \label{Vol0} \tag{a}\\
        &\Vol(\Delta) &&= \sum_x \#\epsilon(x) = |\blacktriangle(\X)|, \label{Vol1} \tag{b}\\
        &\Vol_2(\Delta) &&= \sum_{x} \#\epsilon(x) \cdot \dim \epsilon(x). \label{Vol2} \tag{c}
    \end{alignat}
\end{cor}

\begin{proof}
    The general formula follows immediately from Proposition~\ref{prop:gen function}; the only statement requiring proof is the second equality in~\eqref{Vol1}.
    By~\eqref{epsilon} we have 
    \[
        \#\epsilon(x) = \begin{cases}
            \deg_{\X}(x), & x \in \Min(\X) \cup \Med(\X),\\
            d+1-\deg_{\X}(x), & x \in \Maj(\X),
        \end{cases}
    \]
    in which both cases take the common form $\#\epsilon(x) = \min\{\deg_{\X}(x), \: d+1-\deg_{\X}(x)\}$.
    Hence by Definition~\ref{def:GSD}, we see that $\#\epsilon(x)$ is the multiplicity of $x$ in $\blacktriangle(\X)$, and the result follows.
\end{proof}

We conclude this section with the following corollary, which is the motivating fact behind the definition (and the name) of the EM simplex:

\begin{cor}
    \label{cor:EMD equals volume of Delta}

    Let $h_0, \ldots, h_d$ be histograms, and let $\H = \{H_0, \ldots, H_d\}$ be the family of cumulative histograms.
    Then
    \[
        \EMD(h_0, \ldots, h_d) = \Vol(\Delta(\H)), 
    \]
    where $\Delta(\H)$ is the EM simplex on $\H$.
\end{cor}

\begin{proof}
    This is immediate upon comparing Proposition~\ref{prop:EMC def original} with Corollary~\ref{cor:Vol IDs}\eqref{Vol1}.
\end{proof}

\section{Main results}

We now present our analogue of the Cayley--Menger formula,
expressing the volume of an EM simplex in terms of its edge lengths (i.e., the volumes of its $1$-dimensional faces).
We then present another volume formula, this time in terms of surface area, which, in even dimensions, is an exact generalization of the observation~\eqref{observation 3 and 2} for $d=2$.
For each of our two formulas, we also record a corollary translating the result into the language of the EMD (by means of Corollary~\ref{cor:EMD equals volume of Delta} above), which was the original motivation behind this paper.

\begin{theorem}[Cayley--Menger formula for EM simplices]
    \label{thm:CM}
    
    Let $\Delta$ be an EM simplex of dimension $d$.  
    Then we have
    \[
        \Vol(\Delta) = \frac{1}{d} \Bigg( \Vol_2(\Delta) + \sum_{\mathclap{\substack{\textup{edges $\mathcal{E}$} \\ \textup{of $\Delta$}}}} \Vol(\mathcal{E}) \Bigg).
    \]
\end{theorem}

\begin{proof}
    By Corollary~\ref{cor:Vol IDs}\eqref{Vol1}, we have
    \begin{align*}
        d \cdot \Vol(\Delta) &= \sum_{x \in \cup\X} \#\epsilon(x) \cdot d \\
        &= \sum_x \#\epsilon(x) \cdot \Big(\dim \epsilon(x) + \operatorname{codim} \epsilon(x) \Big)\\
        &= \sum_x \#\epsilon(x) \cdot \dim \epsilon(x) + \sum_x \#\epsilon(x) \cdot \operatorname{codim} \epsilon(x) \\
        &= \Vol_2(\Delta) + \sum_x \underbrace{\#\epsilon(x) \cdot \operatorname{codim} \epsilon(x)}_{\substack{\text{\# edges connecting} \\ \text{$\epsilon(x)$ and its link}}}, \\
    \end{align*}
    where we used Corollary~\ref{cor:Vol IDs}\eqref{Vol2} to rewrite the first sum as a generalized volume.
    As for our claim below the second sum, any such edge clearly has $\#\epsilon(x)$ choices for its first vertex, and $\operatorname{codim} \epsilon(x)$ choices for its second vertex.
    Now, for any edge $\mathcal{E} = \{X_i, X_j\}$, we have $x \in X_i \smtriangle X_j$ if and only if $x$ belongs to exactly one of the two sets, i.e., $\mathcal{E}$ connects $\epsilon(x)$ to its link.
    Therefore, summing over all $x \in \cup\X$, our computation above becomes
    \begin{align*}
        d \cdot \Vol(\Delta) &= \Vol_2(\Delta) + \sum_{i < j} | X_i \smtriangle X_j | \\
        &= \Vol_2(\Delta) + \sum_{\mathclap{\substack{\textup{edges $\mathcal{E}$} \\ \textup{of $\Delta$}}}} \Vol(\mathcal{E}),
    \end{align*}
    where we have rewritten the symmetric differences as edge lengths, using part \eqref{ex: edge length fact} of Example~\ref{ex:volumes}.
\end{proof}

\begin{cor}
    \label{cor:CM}
    Let $h_0, \ldots, h_d$ be histograms, and let $\H$ be the family of cumulative histograms.
    Then
    
    \resizebox{.98\linewidth}{!}{
    \begin{minipage}{\linewidth}
    \begin{align*}
        \EMD(h_0, \ldots, h_d) = \frac{1}{d} \left( \sum_{0 \leq i < j \leq d} \hspace{-1ex} \EMD(h_i, h_j) + \sum_{k=2}^{\lceil d/2 \rceil} k(k-1) \cdot \#\Big\{ x \in \cup\H \;\Big| \; \deg_{\H}(x) = \textup{$k$ or $d-k+1$}\Big\} \right).
    \end{align*}
    \end{minipage}
    }
\end{cor}

Our second main result is a volume formula in terms of surface area rather than edge lengths:

\begin{theorem}[Volume via surface area]
    \label{thm:Vol from SA}

    Let $\Delta = \Delta(\X)$ be an EM simplex of dimension $d$.
    Then
    \[
        \Vol(\Delta) = 
        \frac{1}{d} \left( {\rm{SA}}(\Delta) + \frac{d+1}{2} \cdot |\Med(\X)| \right),
    \]
    where ${\rm SA}(\Delta)$ denotes the surface area (i.e., the sum of the volumes of the facets).
\end{theorem}

\begin{proof}
    In this proof, we reserve the symbol ``$\F$'' specifically for facets of $\Delta$, rather than generic faces.
    We will write $\epsilon_{_{\F}}(x)$ to denote the map $\epsilon$ in the context of $\Delta(\F)$ rather than $\Delta(\X)$: hence the conditions in the definition~\eqref{epsilon} pertain to the membership of $x$ in $\Min(\F)$, $\Maj(\F)$, or $\Med(\F)$.
    We begin by expressing $\#\epsilon(x) \cdot d$ in terms of $\#\epsilon_{_{\F}}(x)$, for each $x \in \cup\X$:

    \textbf{Case 1:} $x \in \Min(\X)$.  We have 
        \begin{align}
            \#\epsilon(x) \cdot d &= \#\epsilon(x) \cdot \Big(\dim \epsilon(x) + \operatorname{codim} \epsilon(x) \Big) \nonumber \\
            &= \#\epsilon(x) \cdot \dim \epsilon(x) + \operatorname{codim} \epsilon(x) \cdot \#\epsilon(x), \label{expansion}
        \end{align}
        which we interpret combinatorially (from left to right) as follows.
        First, $\#\epsilon(x)$ is the number of facets $\F$ whose opposite vertex contains $x$; for any such $\F$, we have $\#\epsilon_{_{\F}}(x) = \#\epsilon(x) - 1 = \dim \epsilon(x)$.
        Next, $\operatorname{codim} \epsilon(x)$ is the number of facets $\F$ whose opposite vertex does \emph{not} contain $x$; for any such $\F$, we have $\#\epsilon_{_{\F}}(x) = \#\epsilon(x)$, since $x \in \Min(\X) \Longrightarrow x \in \Min(\F) \cup \Med(\F)$.
        It follows that the expansion~\eqref{expansion} is the sum of the $\#\epsilon_{_{\F}}(x)$'s, taken over all facets $\F$:
        \begin{equation}
        \label{epsilon(x) d for MinX}
            x \in \Min(\X) \Longrightarrow \#\epsilon(x) \cdot d = \sum_\F \#\epsilon_{_{\F}}(x).
        \end{equation}
        
        \textbf{Case 2:} $x \in \Maj(\X)$.
        In this case, we have a ``dual'' interpretation of~\eqref{expansion}, as follows (from right to left).
        First, $\operatorname{codim} \epsilon(x)$ is the number of facets $\F$ whose opposite vertex contains $x$; for any such $\F$, we have $\#\epsilon_{_{\F}}(x) = \#\epsilon(x)$.
        This is because $x \in \Maj(\X) \Longrightarrow x \in \Med(\F) \cup \Maj(\F)$, and so $\epsilon_{_{\F}}(x)$ equals $\#\F - (\deg_{\X}(x) - 1) = d + 1 - \deg_{\X}(x) = \#\epsilon(x)$.
        Next, $\#\epsilon(x)$ is the number of facets $\F$ whose opposite vertex does \emph{not} contain $x$; for any such $\F$, we have $\#\epsilon_{_{\F}}(x) = \#\epsilon(x) -1 = \dim \epsilon(x)$.
        Therefore, just as in Case 1, the expansion~\eqref{expansion} is the sum of all $\#\epsilon_{_\F}(x)$'s:
        \begin{equation}
        \label{epsilon(x) d for MajX}
            x \in \Maj(\X) \Longrightarrow \#\epsilon(x) \cdot d = \sum_\F \#\epsilon_{_{\F}}(x).
        \end{equation}
        
        \textbf{Case 3:} $x \in \Med(\X)$.
        This time, for each of the $d+1$ facets $\F$, we have $\#\epsilon_{_{\F}}(x) = \frac{d+1}{2} - 1$, regardless of whether $x$ is contained in the opposite vertex.
        Hence
        \[
            \sum_\F \#\epsilon_{_{\F}}(x) = (d+1)\left(\frac{d+1}{2}-1\right) = d \cdot \underbrace{\left(\frac{d+1}{2}\right)}_{\#\epsilon(x)} -\frac{d+1}{2},
        \]
        which we rearrange to obtain
        \begin{equation}
            \label{epsilonx for MedX}
            x \in \Med(\X) \Longrightarrow \#\epsilon(x) \cdot d = \sum_\F \#\epsilon_{_{\F}}(x) + \frac{d+1}{2}.
        \end{equation}
        
        Combining these results, we conclude that
        \begin{align*}
            d \cdot \overbrace{\Vol(\Delta)}^{\text{Cor.~\ref{cor:Vol IDs}\eqref{Vol1}}} = \sum_{x \in \cup\X} \# \epsilon(x) \cdot d & = \overbrace{\sum_{\mathclap{x \in \Min(\X)}} \#\epsilon(x) \cdot d}^{\eqref{epsilon(x) d for MinX}} + \overbrace{\sum_{\mathclap{x \in \Maj(\X)}} \#\epsilon(x) \cdot d}^{\eqref{epsilon(x) d for MajX}} +
            \overbrace{\sum_{\mathclap{x \in \Med(\X)}} \#\epsilon(x) \cdot d}^{\eqref{epsilonx for MedX}}\\
            &= \sum_{x \not\in \Med(\X)} 
            \left(\sum_\F \#\epsilon_{_{\F}}(x) \right) + \sum_{x \in \Med(\X)} \left( \sum_\F \#\epsilon_{_{\F}}(x) + \frac{d+1}{2} \right) \\
            &= \sum_{x \in \cup\X} \left(\sum_\F \#\epsilon_{_{\F}}(x) \right) + \sum_{\mathclap{x \in \Med(\X)}} (d+1)/2\\
            &= \sum_\F \underbrace{\left(\sum_{x \in \cup\F} \#\epsilon_{_{\F}}(x) \right)}_{\Vol(\F)} \hspace{1ex} + \sum_{\mathclap{x \in \Med(\X)}} (d+1)/2\\
            &= {\rm SA}(\Delta) + \frac{d+1}{2} \cdot |\Med(\X)|. \qedhere
        \end{align*}
\end{proof}

\begin{cor}
\label{cor:Heron}
    Let $h_0, \ldots, h_d$ be histograms, and let $\H$ be the family of cumulative histograms.
    Then
    \[
        \EMD(h_0, \ldots, h_d) = \frac{1}{d} \sum_{i=0}^d \EMD(h_0, \ldots, \widehat{h}_i, \ldots, h_d) + \begin{cases}
            0, & \textup{$d$ even},\\
            \frac{d+1}{2d} \cdot |\Med(\H)|, & \textup{$d$ odd}.
        \end{cases}
    \]
\end{cor}

\section{Full example}
\label{sec:Examples}

We revisit the example from Figure~\ref{fig:example EMC d3}, where $d=3$, with the $h_i$ and $H_i$ below:
\input{EMC_small_ex}

\noindent We have already seen that $\EMD(h_0, \ldots, h_3) = 7$; now we will verify Theorems~\ref{thm:CM} and~\ref{thm:Vol from SA}.
This example may also help to clarify the definitions introduced in Section~\ref{sec:EM simplex}.

Let $\H = \{H_0, \ldots, H_3\}$, and let $\Delta = \Delta(\H)$ be the EM simplex.
First we note that $\cup\H$ has 8 elements, namely all dots in the $3 \times 3$ grid except the dot in the upper-left, which does not appear in any of the $H_i$.
We have the following partition of $\cup\H$:
\input{MinMedMaj}

\noindent Next we determine the faces $\epsilon(x)$ for each dot $x \in \cup\H$, maintaining the $3 \times 3$ arrangement from above:
\[
\begin{array}{lll}
     & \{H_1\} & \varnothing \\
  \{H_0\} & \{H_0, H_1\} & \varnothing \\
  \{H_0, H_2\} & \{H_3\} & \varnothing
\end{array}
\]
Below, recalling that the labeling $\lambda$ is just the preimage of $\epsilon$, we depict each set $\lambda(\F)$ by a dot diagram placed at the face $\F \in {\rm skel}_1(\Delta)$.
Note that these faces are either $\varnothing$, vertices, or edges.
If $\lambda(\F) = \varnothing$, then we omit the label on $\F$:
\input{tetrahedron_lambda}

\noindent To compute $\Vol(\Delta)$, we depict the labeling $\lambda'$ below.
A copy of each previous label $\lambda(\F)$ is sent to each facet of $\F$.
Note that $\varnothing$ is the only facet of a vertex:
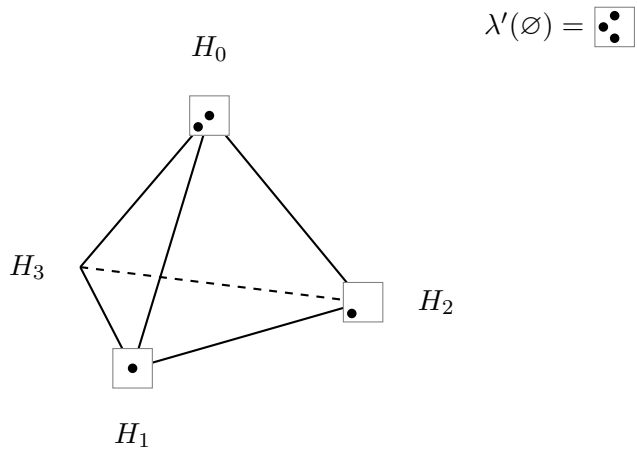
\input{tetrahedron_lambda_prime}

\noindent This verifies Corollary~\ref{cor:EMD equals volume of Delta}: counting dots above, we have $\Vol(\Delta) \coloneqq \sum_\F |\lambda'(\F)| = 7$, which was the EMD computed in Figure~\ref{fig:example EMC d3}.
To deterine $\Vol_2(\Delta)$, we determine one more labeling $\lambda''$, by iterating the process above, namely, sending one copy of $\lambda'(\F)$ to each facet of $\F$.
Clearly $\lambda''(\F)$ is nonempty only for $\F = \varnothing$, in which case we have the multiset union of the three labels shown above:
\[
    \begin{tikzpicture}[scale=.4]
  \foreach \x/\y in {1/1, 2/2} {
    \node[fill, circle, text = white, inner sep = 1pt] at (\x, \y) {\textbf{2}};
  }
  \draw[gray] (0.25,0.25) rectangle (3.75,3.75);
  \node [anchor=east] at (current bounding box.west) {$\lambda''(\varnothing) = $};
  
\end{tikzpicture}
\]
where each dot occurs with multiplicity 2.
Hence $\Vol_2(\Delta) \coloneqq \sum_\F |\lambda''(\F)| = 4$.

Now to verify Theorem~\ref{thm:CM}, we compute the edge lengths of $\Delta$, which is easily done by inspecting the pairwise symmetric differences of the $H_i$, as in Figure~\ref{fig:example EMD d2}:
\input{tetrahedron_edges}

\noindent The sum of the edge lengths is 17.
We therefore have $\frac{1}{d}[\Vol_2(\Delta) + \sum_{\mathcal{E}} \Vol(\mathcal{E})] = \frac{1}{3}(4 + 17) = 7$, which we know to equal $\Vol(\Delta) = \EMD(h_0, \ldots, h_3)$, just as predicted by Theorem~\ref{thm:CM}.

Finally, to verify Theorem~\ref{thm:Vol from SA}, we compute the areas of the four facets of $\Delta$ by inspecting the sizes of the generalized symmetric differences $\blacktriangle$, as in Corollary~\ref{cor:Vol IDs}\eqref{Vol1}:
\begin{align*}
    \Vol(\{H_1, H_2, H_3\}) & = 4,\\
    \Vol(\{H_0, H_2, H_3\}) &= 4,\\
    \Vol(\{H_0, H_1, H_3\}) &= 5,\\
    \Vol(\{H_0, H_1, H_2\}) &= 4.
\end{align*}
Hence $\Delta$ has surface area 17.
Recalling from above that $|\Med(\H)| = 2$, we thus again obtain $\frac{1}{d}[{\rm SA}(\Delta) + \frac{d+1}{2}\cdot |\Med(\H)|] = \frac{1}{3}(17 + 2 \cdot 2) = 7$, agreeing with Theorem~\ref{thm:Vol from SA}.

\begin{rem}
    
    We kept this example three-dimensional in order to include the pictures, but note the price we pay: due to the relationship~\eqref{observation 3 and 2}, since each edge is contained in exactly two facets, the sum of the (2-dimensional) facet areas is the same as the sum of the edge lengths.
    It is therefore much more interesting to contrast the results of Theorems~\ref{thm:CM} and~\ref{thm:Vol from SA} in \emph{more} than three dimensions.

\end{rem}

\section{Concluding remarks and open problems}
\label{sec:Conclusion}

In this paper, we introduced the notion of an EM simplex in order to solve one very specific problem, namely, finding an expression for the generalized EMD in terms of pairwise EMDs.
As we have seen, it turns out that among the generalized volumes $\Vol_i(\Delta)$ with which we endowed an EM simplex, only $\Vol(\Delta) = \Vol_1(\Delta)$ and $\Vol_2(\Delta)$ were required to solve our problem.
It was straightforward to see that $\Vol_0(\Delta)$ is just the size of the union $\cup\X$ of the vertex sets $X_i$, and that $\Vol(\Delta)$ equals the EMD of the histograms corresponding to the vertices.
As a consequence of Theorem~\ref{thm:CM}, we can now interpret $\Vol_2(\Delta)$ as the ``obstruction'' to a perfect generalization of the simple relationship~\eqref{observation 3 and 2} observed previously in dimension 2.

\begin{problem}
    Find an interpretation and/or an application for the higher generalized volumes $\Vol_i(\Delta)$ for $i \geq 3$.
\end{problem}

We would be remiss to conclude without mentioning the field of topological data analysis.
For an EM simplex $\Delta$, it is clear that $\F \subseteq \mathcal{G} \in \Delta$ implies $\Vol(\F) \leq \Vol(\mathcal{G})$.
This is also obvious from the very definition of the generalized EMD, since it is impossible that including more histograms can lessen the work required to equalize them.
Therefore, given an arbitrary number of histograms, the EMD induces a \emph{filtration} on the associated simplicial complex, in the sense of persistent homology: roughly speaking, the EMD functions as a threshold for determining whether a given subset of histograms forms a face.
We are hopeful that this may prove to be a fruitful direction for future research.

\bibliographystyle{alpha}
\bibliography{references}

\end{document}

%% file: EMD_ex.tex
\begin{tikzpicture}[scale=.2]
  \foreach \x/\y in {1/1, 1/2, 1/3, 2/1, 2/2, 2/3, 3/1, 3/2, 3/3, 3/4, 4/1, 4/2, 4/3, 4/4, 4/5, 4/6, 4/7, 4/8, 5/1, 5/2, 5/3, 5/4, 5/5, 5/6, 5/7, 5/8, 5/9, 5/10} {
    \fill[red!40!gray] (\x, \y) circle (10pt);
  }
  
  \draw[->] (0,0) -- (6,0) node[right] {};
  \draw[->] (0,0) -- (0,10) node[above] {};
  
\node [anchor=east] at (current bounding box.west) {$H_0=$};
\node [anchor=south] at (current bounding box.north) {$h_0 = (3,0,1,4,2)$};
  
\end{tikzpicture}
\qquad
\begin{tikzpicture}[scale=.2]
  \foreach \x/\y in {1/1, 2/1, 2/2, 2/3, 2/4, 2/5, 3/1, 3/2, 3/3, 3/4, 3/5, 3/6, 4/1, 4/2, 4/3, 4/4, 4/5, 4/6, 4/7, 5/1, 5/2, 5/3, 5/4, 5/5, 5/6, 5/7, 5/8, 5/9, 5/10} {
    \fill[blue!40!gray] (\x, \y) circle (10pt);
  }
  
  \draw[->] (0,0) -- (6,0) node[right] {};
  \draw[->] (0,0) -- (0,10) node[above] {};

\node [anchor=east] at (current bounding box.west) {$H_1=$};
\node [anchor=south] at (current bounding box.north) {$h_1 = (1,4,1,1,3)$};
  
\end{tikzpicture}
\qquad
\begin{tikzpicture}[scale=.2]
  \foreach \x/\y in {1/2, 1/3, 4/8} {
    \fill[red!40!gray] (\x, \y) circle (10pt);
    }
  \foreach \x/\y in {2/4, 2/5, 3/5, 3/6} {
    \fill[blue!40!gray] (\x, \y) circle (10pt);
    }
  \foreach \x/\y in {1/1,2/1,2/2,2/3,3/1,3/2,3/3,3/4,4/1,4/2,4/3,4/4,4/5,4/6,4/7,5/1,5/2,5/3,5/4,5/5,5/6,5/7,5/8,5/9,5/10} {
    \draw[gray] (\x, \y) circle (10pt);
  }
  
  \draw[->] (0,0) -- (6,0) node[right] {};
  \draw[->] (0,0) -- (0,10) node[above] {};

\node [anchor=east] at (current bounding box.west) {$H_0 \smtriangle H_1=$};
\node [anchor=south] at (current bounding box.north) {$\EMD(h_0, h_1) = 7$};
  
\end{tikzpicture}

%% file: EMC_small_ex.tex
\[
\begin{tikzpicture}[scale=.2]
  \foreach \x/\y in {1/1, 1/2, 2/1, 2/2, 3/1, 3/2, 3/3} {
    \fill (\x, \y) circle (10pt);
  }
  
  \draw[->] (0,0) -- (4,0) node[right] {};
  \draw[->] (0,0) -- (0,4) node[above] {};
  
\node [anchor=east] at (current bounding box.west) {$H_0=$};
\node [anchor=south] at (current bounding box.north) {$h_0 = (2,0,1)$};
  
\end{tikzpicture}
\qquad
\begin{tikzpicture}[scale=.2]
  \foreach \x/\y in {2/1, 2/2, 2/3, 3/1, 3/2, 3/3} {
    \fill (\x, \y) circle (10pt);
  }
  
  \draw[->] (0,0) -- (4,0) node[right] {};
  \draw[->] (0,0) -- (0,4) node[above] {};
  
\node [anchor=east] at (current bounding box.west) {$H_1=$};
\node [anchor=south] at (current bounding box.north) {$h_1 = (0,3,0)$};
  
\end{tikzpicture}
\qquad
\begin{tikzpicture}[scale=.2]
  \foreach \x/\y in {1/1, 2/1, 3/1, 3/2, 3/3} {
    \fill (\x, \y) circle (10pt);
  }
  
  \draw[->] (0,0) -- (4,0) node[right] {};
  \draw[->] (0,0) -- (0,4) node[above] {};
  
\node [anchor=east] at (current bounding box.west) {$H_2=$};
\node [anchor=south] at (current bounding box.north) {$h_2 = (1,0,2)$};
  
\end{tikzpicture}
\qquad
\begin{tikzpicture}[scale=.2]
  \foreach \x/\y in {3/1, 3/2, 3/3} {
    \fill (\x, \y) circle (10pt);
  }
  
  \draw[->] (0,0) -- (4,0) node[right] {};
  \draw[->] (0,0) -- (0,4) node[above] {};
  
\node [anchor=east] at (current bounding box.west) {$H_3=$};
\node [anchor=south] at (current bounding box.north) {$h_3 = (0,0,3)$};
  
\end{tikzpicture}
\]

%% file: EMC_USD_small.tex
\[
\begin{tikzpicture}[scale=.6]
  \foreach \x/\y in {1/3, 3/1, 3/2, 3/3} {
    \node[circle,draw,gray,inner sep=2pt] at (\x, \y) {\phantom{0}};;
  }
  
  \foreach \x/\y in {1/2, 2/3, 2/1} {
    \node[circle,fill,inner sep=2pt,text=white] at (\x, \y) {\textbf{1}};
  }

  \foreach \x/\y in {1/1, 2/2} {
    \node[circle,fill,inner sep=2pt,, text=white] at (\x, \y) {\textbf{2}};;
  }
  
  \draw[->] (0,0) -- (4,0) node[right] {};
  \draw[->] (0,0) -- (0,4) node[above] {};

\node [anchor=east] at (current bounding box.west) {$\blacktriangle(H_0, \ldots, H_3) =$};
\node at (2,-1) {$\EMD(h_0, \ldots, h_3) = 7$};
  
\end{tikzpicture}
\]

%% file: MinMedMaj.tex
\[
\begin{tikzpicture}[scale=.2]
  \foreach \x/\y in {1/2, 2/3} {
    \fill (\x, \y) circle (10pt);
  }
  \draw[gray] (0.25,0.25) rectangle (3.75,3.75);
  \node [anchor=south] at (current bounding box.north) {$\Min(\H)$};
  
\end{tikzpicture} 
\qquad
\begin{tikzpicture}[scale=.2]
  \foreach \x/\y in {1/1, 2/2} {
    \fill (\x, \y) circle (10pt);
  }
  \draw[gray] (0.25,0.25) rectangle (3.75,3.75);
  \node [anchor=south] at (current bounding box.north) {$\Med(\H)$};
  
\end{tikzpicture}
\qquad
\begin{tikzpicture}[scale=.2]
  \foreach \x/\y in {2/1, 3/1, 3/2, 3/3} {
    \fill (\x, \y) circle (10pt);
  }
  \draw[gray] (0.25,0.25) rectangle (3.75,3.75);
  \node [anchor=south] at (current bounding box.north) {$\Maj(\H)$};
  
\end{tikzpicture}
\]

%% file: tetrahedron_lambda.tex
\tdplotsetmaincoords{110}{100} 

\newsavebox\one

\begin{lrbox}{\one}
\begin{tikzpicture}[scale=.15]
  \draw[draw=gray,fill=white] (0.25,0.25) rectangle (3.75,3.75);
  \foreach \x/\y in {2/3} {
    \fill (\x, \y) circle (12pt);
  }
  \end{tikzpicture}
\end{lrbox}

\newsavebox\zero

\begin{lrbox}{\zero}
\begin{tikzpicture}[scale=.15]
  \draw[draw=gray,fill=white] (0.25,0.25) rectangle (3.75,3.75);
  \foreach \x/\y in {1/2} {
    \fill (\x, \y) circle (12pt);
  }
  \end{tikzpicture}
\end{lrbox}

\newsavebox\three

\begin{lrbox}{\three}
\begin{tikzpicture}[scale=.15]
  \draw[draw=gray,fill=white] (0.25,0.25) rectangle (3.75,3.75);
  \foreach \x/\y in {2/1} {
    \fill (\x, \y) circle (12pt);
  }
  \end{tikzpicture}
\end{lrbox}

\newsavebox\zeroone

\begin{lrbox}{\zeroone}
\begin{tikzpicture}[scale=.15]
  \draw[draw=gray,fill=white] (0.25,0.25) rectangle (3.75,3.75);
  \foreach \x/\y in {2/2} {
    \fill (\x, \y) circle (12pt);
  }
  \end{tikzpicture}
\end{lrbox}

\newsavebox\zerotwo

\begin{lrbox}{\zerotwo}
\begin{tikzpicture}[scale=.15]
  \draw[draw=gray,fill=white] (0.25,0.25) rectangle (3.75,3.75);
  \foreach \x/\y in {1/1} {
    \fill (\x, \y) circle (12pt);
  }
  \end{tikzpicture}
\end{lrbox}

\newsavebox\nullset

\begin{lrbox}{\nullset}
\begin{tikzpicture}[scale=.15]
  \draw[draw=gray,fill=white] (0.25,0.25) rectangle (3.75,3.75);
  \foreach \x/\y in {3/1, 3/2, 3/3} {
    \fill (\x, \y) circle (12pt);
  }
  \node [anchor=east] at (current bounding box.west) {$\lambda(\varnothing) = $};
  \end{tikzpicture}
\end{lrbox}
  
\[
\begin{tikzpicture}[tdplot_main_coords, scale=2, thick, label distance=5pt]
  \coordinate (1) at (0,0,0);
  \coordinate (3) at (2,0,0);
  \coordinate (2) at ({2*cos(60)},{2*sin(60)},0);
  \coordinate (0) at ({(2*cos(60))/3},{(2*sin(60))/3},{2*sqrt(6)/3});
  
  \draw (3) -- (1) -- (2);
  \draw[dashed] (3) -- (2);
  \draw (1) -- (0) node [midway] {\usebox\zeroone};
  \draw (3) -- (0);
  \draw (2) -- (0) node [midway] {\usebox\zerotwo};
  
  \node[label={270:{$H_1$}}] at (1) {\usebox\one};
  \node[label={180:{$H_3$}}] at (3) {\usebox\three};
  \node[label={0:{$H_2$}}] at (2) {};
  \node[label={90:{$H_0$}}] at (0) {\usebox\zero};

  \node [anchor=west] at (current bounding box.north east) {\usebox\nullset};
\end{tikzpicture}
\]

%% file: tetrahedron_lambda_prime.tex
\tdplotsetmaincoords{110}{100} 

\newsavebox\primezero

\begin{lrbox}{\primezero}
\begin{tikzpicture}[scale=.15]
  \draw[draw=gray,fill=white] (0.25,0.25) rectangle (3.75,3.75);
  \foreach \x/\y in {1/1, 2/2} {
    \fill (\x, \y) circle (12pt);
  }
  \end{tikzpicture}
\end{lrbox}

\newsavebox\primeone

\begin{lrbox}{\primeone}
\begin{tikzpicture}[scale=.15]
  \draw[draw=gray,fill=white] (0.25,0.25) rectangle (3.75,3.75);
  \foreach \x/\y in {2/2} {
    \fill (\x, \y) circle (12pt);
  }
  \end{tikzpicture}
\end{lrbox}

\newsavebox\primetwo

\begin{lrbox}{\primetwo}
\begin{tikzpicture}[scale=.15]
  \draw[draw=gray,fill=white] (0.25,0.25) rectangle (3.75,3.75);
  \foreach \x/\y in {1/1} {
    \fill (\x, \y) circle (12pt);
  }
  \end{tikzpicture}
\end{lrbox}

\newsavebox\primenullset

\begin{lrbox}{\primenullset}
\begin{tikzpicture}[scale=.15]
  \draw[draw=gray,fill=white] (0.25,0.25) rectangle (3.75,3.75);
  \foreach \x/\y in {1/2, 2/1, 2/3} {
    \fill (\x, \y) circle (12pt);
  }
  \node [anchor=east] at (current bounding box.west) {$\lambda'(\varnothing) = $};
  \end{tikzpicture}
\end{lrbox}
  
\[
\begin{tikzpicture}[tdplot_main_coords, scale=2, thick, label distance=5pt]
  \coordinate (1) at (0,0,0);
  \coordinate (3) at (2,0,0);
  \coordinate (2) at ({2*cos(60)},{2*sin(60)},0);
  \coordinate (0) at ({(2*cos(60))/3},{(2*sin(60))/3},{2*sqrt(6)/3});
  
  \draw (3) -- (1) -- (2);
  \draw[dashed] (3) -- (2);
  \draw (1) -- (0);
  \draw (3) -- (0);
  \draw (2) -- (0);
  
  \node[label={270:{$H_1$}}] at (1) {\usebox\primeone};
  \node[label={180:{$H_3$}}] at (3) {};
  \node[label={0:{$H_2$}}] at (2) {\usebox\primetwo};
  \node[label={90:{$H_0$}}] at (0) {\usebox\primezero};

  \node [anchor=west] at (current bounding box.north east) {\usebox\primenullset};
\end{tikzpicture}
\]

%% file: tetrahedron_edges.tex
\tdplotsetmaincoords{110}{100} 

\[
\begin{tikzpicture}[tdplot_main_coords, scale=2, thick, label distance=5pt]
  \coordinate (1) at (0,0,0);
  \coordinate (3) at (2,0,0);
  \coordinate (2) at ({2*cos(60)},{2*sin(60)},0);
  \coordinate (0) at ({(2*cos(60))/3},{(2*sin(60))/3},{2*sqrt(6)/3});
  
  \draw (3) -- (1) node[midway, left] {3};
  \draw (1) -- (2) node[midway, below] {3};
  \draw[dashed] (3) -- (2) node[midway, above] {2};
  \draw (1) -- (0) node[pos=.6, left] {3};
  \draw (3) -- (0) node[midway, left] {4};
  \draw (2) -- (0) node[midway, right] {2};
  
  \node[label={270:{$H_1$}}] at (1) {};
  \node[label={180:{$H_3$}}] at (3) {};
  \node[label={0:{$H_2$}}] at (2) {};
  \node[label={90:{$H_0$}}] at (0) {};
\end{tikzpicture}
\]